\documentclass[a4paper,english,fontsize=11pt,parskip=half]{scrartcl}
\usepackage{babel}
\usepackage[utf8]{inputenc}
\usepackage[T1]{fontenc}
\usepackage[a4paper,left=20mm,right=20mm,top=30mm,bottom=30mm]{geometry}
\usepackage{amsmath}
\usepackage{amsthm}
\usepackage{amssymb}
\usepackage{enumerate}
\usepackage{thmtools}
\usepackage{mathtools}
\mathtoolsset{centercolon}
\usepackage[bookmarks=false,
            pdftitle={Fusion systems in representation theory},
            pdfauthor={Benjamin Sambale},
            pdfkeywords={introduction, fusion systems, finite groups, blocks},
            pdfstartview={FitH}]{hyperref}

\newtheorem{Thm}{Theorem}[section]

\newtheorem{Cor}[Thm]{Corollary}
\newtheorem{Con}[Thm]{Conjecture}
\theoremstyle{definition}
\newtheorem{Def}[Thm]{Definition}
\newtheorem{Ex}[Thm]{Example}

\numberwithin{equation}{section}

\allowdisplaybreaks[1]

\renewcommand{\phi}{\varphi}
\newcommand{\C}{\mathrm{C}}
\newcommand{\N}{\mathrm{N}}
\newcommand{\Z}{\mathrm{Z}}
\newcommand{\FC}{\mathcal{F}}
\newcommand{\EC}{\mathcal{E}}
\newcommand{\UC}{\mathcal{U}}
\newcommand{\ZZ}{\mathbb{Z}}

\newcommand{\FF}{\mathbb{F}}

\newcommand{\cohom}{\operatorname{H}}
\newcommand{\Aut}{\mathrm{Aut}}
\newcommand{\Inn}{\mathrm{Inn}}
\newcommand{\Out}{\mathrm{Out}}
\newcommand{\Sym}{\operatorname{Sym}}
\newcommand{\pcore}{\mathrm{O}}
\newcommand{\GL}{\mathrm{GL}}
\newcommand{\ASL}{\mathrm{ASL}}
\newcommand{\SL}{\operatorname{SL}}
\newcommand{\PSU}{\operatorname{PSU}}
\newcommand{\PSL}{\operatorname{PSL}}
\newcommand{\PGL}{\operatorname{PGL}}

\newcommand{\Sz}{\operatorname{Sz}}
\newcommand{\Irr}{\mathrm{Irr}}
\newcommand{\IBr}{\mathrm{IBr}}

\newcommand{\J}{\mathrm{J}}
\newcommand{\ZJ}{\mathrm{ZJ}}
\newcommand{\id}{\mathrm{id}}

\newcommand{\foc}{\mathfrak{foc}}
\newcommand{\hyp}{\mathfrak{hyp}}

\newcommand{\Syl}{\operatorname{Syl}}

\newcommand{\Hom}{\operatorname{Hom}}

\newcommand{\Spin}{\operatorname{Spin}}
\newcommand{\Obj}{\mathrm{Obj}}

\AtBeginDocument{
  \DeclareFontShape{T1}{cmr}{m}{scit}{<->ssub*cmr/m/sc}{}%
}


\title{Fusion systems in representation theory}
\subtitle{Three lectures at the University of Valencia}
\author{Benjamin Sambale\footnote{Leibniz Universität Hannover, Germany, \href{mailto:sambale@math.uni-hannover.de}{sambale@math.uni-hannover.de}}}
\date{February 2023}

\begin{document}
\frenchspacing
\maketitle

\renewcommand{\sectionautorefname}{Section}
\section{Fusion in groups}\label{sec1}

\begin{Def}
Let $H\le G$ be finite groups. Elements $x,y\in H$ (or subsets) are called \emph{fused} in $G$ if they are conjugate in $G$, but not in $H$. 
\end{Def}

\begin{Ex}\label{exfused}\hfill
\begin{enumerate}[(i)]
\item The permutations $(123),(132)\in A_3$ are fused in $S_3$.

\item\label{ex2} Let $X,Y\le H$ be isomorphic subgroups via an isomorphism $\phi\colon X\to Y$. We embed $H$ into $G:=\Sym(H)$ via the regular representation $\sigma\colon H\to G$, $h\mapsto\sigma_h$ where $\sigma_h(g)=hg$ for $g,h\in H$. 
Let $\{r_1,\ldots,r_n\}$ and $\{s_1,\ldots,s_n\}$ be systems of representatives for the right cosets in $H$ of $X$ and $Y$, respectively.\footnote{Using Hall's marriage theorem one can show that a \emph{common} system of representatives exists for these right cosets. This is not needed here.}
Let $\hat\phi\in G$ with $\hat\phi(xr_i):=\phi(x)s_i$ for all $x\in X$ and $1\le i\le n$.
For $x\in X$, $y\in Y$ and $1\le i\le n$ we have
\[(\hat\phi\sigma_x\hat\phi^{-1})(ys_i)=(\hat\phi\sigma_x)(\phi^{-1}(y)r_i)=\hat\phi(x\phi^{-1}(y)r_i)=\phi(x)ys_i=\sigma_{\phi(x)}(ys_i)\]
and $\hat\phi\sigma_x\hat\phi^{-1}=\sigma_{\phi(x)}$.
Hence, $\phi$ is realized by the conjugation with $\hat{\phi}$ in $G$.\footnote{This construction fails for infinite groups since for example the isomorphism $\ZZ\to 2\ZZ$ does not extend to $\ZZ\to\ZZ$. In those situations one can use HNN-\emph{extensions}.} 

\item A consequence of \eqref{ex2} is that elements $x,y\in H$ of the same order are conjugate in some finite group $G\ge H$. 

\end{enumerate}
\end{Ex}

\textbf{Goal:} Find “small” subgroups $K\supseteq H$ \emph{controlling fusion} in $H$, i.\,e. $x,y\in H$ are fused in $G$ if and only if $x,y$ are fused in $K$.

\textbf{Main interest:} $H\in\Syl_p(G)$.

In the following let $P\in\Syl_p(G)$. Let $\pcore_{p'}(G)$ be the largest normal $p'$-subgroup of $G$.
If no elements of $P$ are fused in $G$, then $G$ is called \emph{$p$-nilpotent}.

\begin{Thm}[\textsc{Frobenius}]\label{thmfrob}
The following assertions are equivalent:
\begin{enumerate}[(1)]
\item $G$ is $p$-nilpotent.
\item\label{frob2} $\N_G(Q)/\C_G(Q)$ is a $p$-group for all $Q\le P$.
\item $G=\pcore_{p'}(G)P$.
\end{enumerate}
\end{Thm}

\begin{Ex}
Every $p'$-group and every nilpotent group is $p$-nilpotent.
\end{Ex}

\begin{Thm}[\textsc{Burnside}]
$\N_G(P)$ controls fusion in $\Z(P)$.
\end{Thm}
\begin{proof}
Let $x,y\in\Z(P)$ and $g\in G$ such that $^gx:=gxg^{-1}=y$. Then $P\le\C_G(y)$ and $^gP\le{}^g\C_G(x)=\C_G({}^gx)=\C_G(y)$. By Sylow's theorem, there exists $c\in\C_G(y)$ such that $^{cg}P=P$. Now $h:=cg\in\N_G(P)$ such that $^hx={}^c({}^gx)={}^cy=y$. 
\end{proof}

\begin{Thm}[$\Z^*$-theorem\protect\footnote{It is often assumed that $z$ has order $p$, but this is unnecessary.}]
If $z\in\Z(P)$ is not fused to any other element in $P$, then $G=\pcore_{p'}(G)\C_G(z)$.
\end{Thm}
\begin{proof}
Glauberman proved the theorem for $p=2$ using representation theory, while the only known proof for $p>2$ is via the classification of finite simple groups (CFSG for short).
\end{proof}

\begin{Ex}
If $P$ is a (generalized) quaternion $2$-group, then $G=\pcore_{2'}(G)\C_G(\Z(P))$ since $\Z(P)$ is generated by the unique involution in $P$.\footnote{This special case of the $\Z^*$-theorem was first proved by Brauer--Suzuki.}
\end{Ex}

Goldschmidt and Flores--Foote classified more generally groups $G$ with $A\unlhd P$ such that no element of $A$ is fused to an element of $P\setminus A$ (i.\,e. $A$ is \emph{strongly closed} in $P$).
Let 
\[\J(P):=\langle A\le P:A\text{ abelian of maximal order}\rangle\] 
be the \emph{Thompson subgroup} of $P$.\footnote{Several non-equivalent definitions of the Thompson subgroup are used in the literature.}

\begin{Thm}[\textsc{Thompson}]
If $p\ge 5$, then $G$ is $p$-nilpotent if and only if $\N_G(\J(P))/\C_G(\J(P))$ is a $p$-group.
\end{Thm}

\begin{Thm}[\textsc{Glauberman}'s $\ZJ$-theorem]\label{ZJ}
Let $p>2$.
Then $G$ is $p$-nilpotent if and only if\linebreak $\N_G\bigl(\Z(\J(P))\bigr)$ is $p$-nilpotent. 
If $G$ has no section isomorphic to $Qd(p):=C_p^2\rtimes\SL_2(p)$, then $\N_G\bigl(\Z(\J(P))\bigr)$ controls fusion in $P$.  
\end{Thm}

\begin{Ex}
For $p\ge 5$, every ($p$-)solvable group is $Qd(p)$-free.
\end{Ex}

\begin{Thm}[\textsc{Stellmacher}]\label{Stellmacher}
If $p=2$ and $G$ has no section isomorphic to $Qd(2)\cong S_4$, then $\N_G(W)$ controls fusion in $P$ for some characteristic subgroup $W$ of $P$. If $P\ne 1$, then $W\ne 1$.
\end{Thm}

Let $G'=[G,G]$ be the commutator subgroup and $\pcore^p(G)=\langle p'\text{-elements}\rangle$ the $p$-residue of $G$.

\begin{Thm}[(Hyper)focal subgroup theorem]
\begin{align*}
\foc_G(P)&:=\langle xy^{-1}:x,y\in P\text{ are conjugate in }G\rangle=G'\cap P\quad(\emph{focal subgroup}),\\
\hyp_G(P)&:=\langle xy^{-1}:x,y\in P\text{ are conjugate by a $p'$-element}\rangle=\pcore^p(G)\cap P\quad(\emph{hyperfocal subgroup}).
\end{align*}
\end{Thm}

The \emph{transfer map} yields $G/\pcore^p(G)\cong P/\hyp_G(P)$. 

\begin{Thm}[\textsc{Grün}'s theorem]
\[\foc_G(P)=[\N_G(P),P]\langle P\cap Q':Q\in\Syl_p(G)\rangle.\]
\end{Thm}

Let $\Phi(P)$ be the Frattini subgroup of $P$.

\begin{Thm}\label{Tatepnil}
The following assertions are equivalent:
\begin{enumerate}[(1)]
\item $G$ is $p$-nilpotent.
\item $\hyp_G(P)=1$.
\item $\hyp_G(P)\le\Phi(P)$.
\end{enumerate}
\end{Thm}

\begin{Thm}[\textsc{Tate}'s transfer theorem]
For $P\le H\le G$ we have
\[\foc_G(P)=\foc_H(P)\iff\hyp_G(P)=\hyp_H(P)\iff\foc_G(P)\Phi(P)=\foc_H(P)\Phi(P).\]
\end{Thm}

If $\foc_G(P)=\foc_H(P)$, we say that $H$ \emph{controls transfer} in $P$. In this case $H$ determines whether $G$ is $p$-nilpotent by \autoref{Tatepnil}. 

\begin{Thm}[\textsc{Yoshida}'s transfer theorem]
If $P$ has no quotient isomorphic to $C_p\wr C_p$, then $\N_G(P)$ controls transfer in $P$.
\end{Thm}

\begin{Ex}\hfill
\begin{enumerate}[(i)]
\item If $|P|\le p^p$ or $\exp(P)=p$ (exponent) or $c(P)<p$ (nilpotency class), then $\N_G(P)$ controls transfer in $P$. This follows from the properties of $C_p\wr C_p$.

\item Let $p=2$ and $G=S_4$. Then $\N_G(P)=P\cong D_8\cong C_2\wr C_2$ does not control transfer in $P$ since otherwise $G$ would be $2$-nilpotent. For $p>2$ and
\[G=\FF_p^p\rtimes\left\langle \begin{pmatrix}
1&&&0\\
&-1\\
&&\ddots\\
0&&&-1
\end{pmatrix},\begin{pmatrix}
0&1&&0\\
&\ddots&\ddots\\
&&\ddots&1\\
1&&&0
\end{pmatrix}\right\rangle\le\ASL(p,p),\]
again $\N_G(P)=P\cong C_p\wr C_p$ does not control transfer in $P$.  
\end{enumerate}
\end{Ex}

\begin{Thm}[\textsc{Glauberman}]\label{glauberman}
If $p\ge 5$, then there exists a characteristic subgroup $K$ of $P$ such that $\N_G(K)$ controls transfer in $P$ and $\Z(P)\le K$.
\end{Thm}

The simple group $\PSL(2,17)$ shows that \autoref{glauberman} fails for $p=2$ (here $P$ is a maximal subgroup). 
It is an open problem whether \autoref{glauberman} holds for $p=3$. 
For $p\ge 7$ one can take $K=\J(P)$. 

\section{Fusion systems}

For arbitrary groups $S,T\le P$ let $\Hom_P(S,T)$ be the set of homomorphisms $S\to T$ induced by inner automorphisms of $P$, i.\,e. 
\[\Hom_P(S,T):=\bigl\{\phi\colon S\to T:\exists g\in P:\phi(s)={}^gs\ \forall s\in S\bigr\}.\]

\begin{Def}[\textsc{Puig}\protect\footnote{Puig calls them \emph{Frobenius categories}}]
A \emph{fusion system} on a finite $p$-group $P$ is a category $\FC$ with objects $\Obj(\FC)=\{S:S\le P\}$ and morphisms $\Hom_\FC(S,T)\subseteq\{S\to T:\text{injective group homomorphism}\}$ such that 
\begin{itemize}
\item $\Hom_P(S,T)\subseteq\Hom_{\FC}(S,T)$ for $S,T\le P$,
\item $\phi\in\Hom_{\FC}(S,T)\ \Longrightarrow\ \phi\in\Hom_{\FC}(S,\phi(S))$, $\phi^{-1}\in\Hom_\FC(\phi(S),S)$.
\end{itemize}
\end{Def}

\begin{Ex}\hfill
\begin{enumerate}[(i)]
\item Let $P$ be a $p$-subgroup of a finite group $G$. Then $\Hom_\FC(S,T):=\Hom_G(S,T)$ for $S,T\le P$ defines a fusion system on $P$, which we denote by $\FC_P(G)$. 
In particular, there is always the \emph{trivial} fusion system $\FC_P(P)$, which is a subcategory of every fusion system on $P$.

\item The \emph{universal} fusion system $\FC:=\UC(P)$ on $P$ is defined by 
\[\Hom_\FC(S,T):=\{S\to T\text{ injective homomorphism}\}.\] 
Every fusion system on $P$ is a subcategory of $\UC(P)$.  
\end{enumerate}
\end{Ex}

\begin{Thm}[\textsc{Park}]\label{park}
For every fusion system $\FC$ on $P$ there exists a finite group $G$ containing $P$ such that $\FC=\FC_P(G)$.
\end{Thm}

\autoref{park} remains true even for arbitrary finite groups $P$ with appropriate definitions (see \autoref{exfused}\eqref{ex2} for $\FC=\UC(P)$).

\begin{Def}
Let $\FC$ be a fusion system on $P$ and $S,T\le P$. 
\begin{itemize}
\item $S,T$ are called \emph{$\FC$-conjugate} if there exists an isomorphism $\phi\colon S\to T$ in $\FC$.
\item $S$ is called \emph{$\FC$-automized} if $\Aut_P(S)\in\Syl_p(\Aut_\FC(S))$. 
\item $S$ is called \emph{$\FC$-centralized}\footnote{often called \emph{fully} $\FC$-centralized/normalized} if $|\C_P(S)|\ge|\C_P(T)|$ for all $\FC$-conjugates $T$ of $S$.
\item $S$ is called \emph{$\FC$-normalized} if $|\N_P(S)|\ge|\N_P(T)|$ for all $\FC$-conjugates $T$ of $S$.
\item For an isomorphism $\phi\colon S\to T$ let $N_\phi$ be the preimage of $\Aut_P(S)\cap \phi^{-1}\Aut_P(T)\phi$ under the conjugation map $\N_P(S)\to\Aut_P(S)$, $x\mapsto c_x$, i.\,e.
\[N_\phi:=\bigl\{x\in\N_P(S):\phi c_x\phi^{-1}\in\Aut_P(T)\bigr\}.\] 
\item $T$ is called \emph{$\FC$-receptive} if every isomorphism $\phi\colon S\to T$ in $\FC$ extends to $N_\phi$ (note that $S\C_P(S)\le N_\phi\le\N_P(S)$).
\end{itemize}
\end{Def}

\begin{Ex}\hfill
\begin{enumerate}[(i)]
\item If $S,T\le P\le G$ are fused in $G$, then they are $\FC_P(G)$-conjugate.
\item If $P\in\Syl_p(G)$, then $P$ is automized in $\FC_P(G)$, because $P\C_G(P)/\C_G(P)\in\Syl_p(\N_G(P)/\C_G(P))$. 
\item Every central subgroup of $P$ is $\FC$-centralized and every normal subgroup is $\FC$-normalized.
\item Every $\FC$-receptive subgroup is $\FC$-centralized: Let $T\le P$ be receptive and $\phi\colon S\to T$ an isomorphism in $\FC$. Then $\phi$ extends to $\hat\phi\colon N_\phi\to P$. For $s\in S$ and $g\in\C_P(S)$ we have $\hat\phi(g)\phi(s)\hat\phi(g)^{-1}=\hat\phi(gsg^{-1})=\phi(s)$ and $\hat\phi(\C_P(S))\le\C_P(T)$. Since morphisms are injective, it follows that $|\C_P(S)|\le|\C_P(T)|$.
\item Every $\FC$-centralized, $\FC$-automized subgroup $S\le P$ is $\FC$-normalized. This follows from $|\N_P(S)|=|\Aut_P(S)||\C_P(S)|$.
\item Let $S:=\langle(12)(34)\rangle\le P:=\langle(1234),(13)\rangle\le G:=S_4$ and $\FC:=\FC_P(G)$. Then $S$ is neither $\FC$-centralized nor $\FC$-normalized since $S$ is $\FC$-conjugate to $\Z(P)=\langle(13)(24)\rangle$. 
\end{enumerate}
\end{Ex}

\begin{Thm}\label{thmequivdef}
The following assertions for a fusion system $\FC$ on $P$ are equivalent:
\begin{enumerate}[(1)]
\item \textup(\textsc{Roberts--Shpectorov}\textup) Every subgroup of $P$ is $\FC$-conjugate to an automized, receptive subgroup.
\item $P$ is automized and every subgroup of $P$ is $\FC$-conjugate to a normalized, receptive subgroup.
\item \textup(\textsc{Stancu}\textup) $P$ is automized and every normalized subgroup of $P$ is receptive. 
\item \textup(\textsc{Broto--Levi--Oliver}\textup) Every normalized subgroup of $P$ is centralized and automized and every centralized subgroup is receptive.
\end{enumerate}
Under these circumstances we call $\FC$ \emph{saturated}.
\end{Thm}

For a saturated fusion system $\FC$ on $P$ and $S\le P$ we have
\begin{enumerate}[(i)]
\item $S$ is $\FC$-centralized if and only if $S$ is $\FC$-receptive.
\item $S$ is $\FC$-normalized if and only if $S$ is $\FC$-centralized and $\FC$-automized. 
\end{enumerate}

\begin{Thm}
If $P\in\Syl_p(G)$, then $\FC_P(G)$ is saturated.
\end{Thm}
\begin{proof}
We prove \autoref{thmequivdef}(1) for $\FC:=\FC_P(G)$.
Let $Q\le P$ and $\N_P(Q)\le R\in\Syl_p(\N_G(Q))$. By Sylow's theorem, there exists $g\in G$ such that 
\[T:={}^gQ\le {}^gR\le P.\] 
Since $^gR\in\Syl_p({}^g\N_G(Q))=\Syl_p(\N_G(T))$, we have $^gR=\N_P(T)$ and $T$ is $\FC$-automized. 

Now let $\phi\colon S\to T$ be an arbitrary isomorphism in $\FC$. Then there exists $a\in G$ with $\phi(s)={}^as$ for all $s\in S$. 
For $x\in N_\phi$ there exists $y\in\N_P(T)$ such that 
\[^{axa^{-1}}t=(\phi c_x\phi^{-1})(t)={}^yt\]
for all $t\in T$. Hence, $y^{-1}axa^{-1}\in\C_G(T)$ and $axa^{-1}\in\N_P(T)\C_G(T)$.
By definition, $N_\phi\le\N_P(S)$ is a $p$-group and $^aN_\phi$ is a $p$-subgroup of $\N_P(T)\C_G(T)$. Since $\N_P(T)$ is a Sylow $p$-subgroup of $\N_G(T)\ge \N_P(T)\C_G(T)$, there exist $h\in\N_P(T)$ and $z\in\C_G(T)$ with $^{hza}N_\phi\le\N_P(T)$.
Then also $^{za}N_\phi \le\N_P(T)\le P$. For $s\in S$ we have $^{za}s={}^z\phi(s)=\phi(s)$. Hence, the conjugation with $za$ is an extension of $\phi$ to $N_\phi$ in $\FC$. Consequently, $T$ is $\FC$-receptive.
\end{proof}

\begin{Ex}
Let $|P|>p$. A theorem of Gaschütz asserts that $P$ has an outer automorphism of $p$-power order. Hence, $P$ is not automized in $\UC(P)$ and $\UC(P)$ is not saturated.
\end{Ex}

\begin{Thm}[\textsc{Robinson}, \textsc{Leary--Stancu}]
For every saturated fusion system $\FC$ on $P$ there exists an \emph{infinite} group $G$ with $P\in\Syl_p(G)$ such that $\FC=\FC_P(G)$. 
\end{Thm}

\begin{Def}
A saturated fusion system $\FC$ is called \emph{exotic} if there is no \emph{finite} group $G$ with $P\in\Syl_p(G)$ and $\FC=\FC_P(G)$. 
\end{Def}

\begin{Ex}\label{exexotic}\hfill
\begin{enumerate}[(i)]
\item For $p=2$ the only known simple exotic fusion systems are defined on the Sylow $2$-subgroups of $\Spin_7(q)\cong 2.\Omega_7(q)$ where $q$ is an odd prime power. These are called the \emph{Solomon fusion systems}. For $q=3$ we have $|P|=2^{10}$.

\item  For $p>2$ many families of exotic fusion systems have been discovered recently. For instance, Ruiz--Viruel constructed an exotic fusion system $\FC$ on the extraspecial group $P$ of order $7^3$ with exponent $7$ such that all non-trivial elements of $P$ are $\FC$-conjugate.
\end{enumerate}
\end{Ex}

Most of the fusion and transfer theorems for finite groups stated in \autoref{sec1} have been translated to fusion systems. For instance, a saturated fusion system $\FC$ is trivial if and only if $\Aut_\FC(Q)$ is a $p$-group for every $Q\le P$. This will be generalized in the next section. To state some more theorems, we need the following constructions.

\begin{Def}
Let $\FC$ be a saturated fusion system on $P$ and $Q\le P$.
\begin{itemize}
\item The fusion system $\C_\FC(Q)$ on $\C_P(Q)$ consists of the morphisms $\phi\colon S\to T$ such that there exists a morphism $\psi\colon QS\to QT$ in $\FC$ with $\psi_S=\phi$ and $\psi_Q=\id_Q$. 

\item The fusion system $\N_\FC(Q)$ on $\N_P(Q)$ consists of the morphisms $\phi\colon S\to T$ such that there exists a morphism $\psi\colon QS\to QT$ in $\FC$ with $\psi_S=\phi$ and $\psi(Q)=Q$. 

\item The fusion system $Q\C_\FC(Q)$ on $Q\C_P(Q)$ consists of the morphisms $\phi\colon S\to T$ such that there exists a morphism $\psi\colon QS\to QT$ in $\FC$ with $\psi_S=\phi$ and $\psi_Q\in\Inn(Q)$. 
\end{itemize}
\end{Def}

Recall that every subgroup $Q\le P$ is $\FC$-conjugate to an $\FC$-normalized subgroup. In this case, Puig has shown that $\C_\FC(Q)$, $\N_\FC(Q)$ and $Q\C_\FC(Q)$ are saturated. 

\begin{Ex}
Let $Q\le P\in\Syl_p(G)$ and $\FC=\FC_P(G)$. If $Q$ is $\FC$-normalized, then $\C_\FC(Q)=\FC_{\C_P(Q)}(\C_G(Q))$, $\N_\FC(Q)=\FC_{\N_P(Q)}(\N_G(Q))$ and $Q\C_\FC(Q)=\FC_{Q\C_P(Q)}(Q\C_G(Q))$. 
\end{Ex}

\begin{Thm}[\textsc{Kessar--Linckelmann}]
A saturated fusion system $\FC$ on $P$ with $p>2$ is trivial if and only if $\N_\FC\bigl(\Z(\J(P))\bigr)$ is trivial.  
\end{Thm}

\begin{Def}
For a saturated fusion system $\FC$ on $P$ we define
\begin{align*}
\Z(\FC)&:=\bigl\{x\in P:\phi(x)=x\ \forall \phi\in\Hom_\FC(\langle x\rangle,P)\bigr\}\qquad(\emph{center}),\\
\foc(\FC)&:=\langle \phi(x)x^{-1}:x\in P,\ \phi\in\Hom_\FC(\langle x\rangle,P)\rangle\qquad(\emph{focal subgroup}),\\
\hyp(\FC)&:=\langle \phi(x)x^{-1}:x\in Q\le P,\ \phi\in\pcore^p(\Aut_\FC(Q))\rangle\qquad(\emph{hyperfocal subgroup}).
\end{align*}
\end{Def}

\begin{Ex}\label{exfoc}\hfill
\begin{enumerate}[(i)]
\item The center $\Z(\FC)$ is the largest subgroup $Q\le P$ such that $\C_\FC(Q)=\FC$.

\item One can show that $\foc(\FC)=\hyp(\FC)P'$ and $\foc(\FC)\cap\Z(\FC)=P'\cap\Z(\FC)$. In particular, the \emph{Fitting decomposition} $P=\Z(\FC)\times\foc(\FC)$ holds whenever $P$ is abelian. 

\item If $\FC=\FC_P(G)$, then $\foc(\FC)=\foc_G(P)$, $\hyp(\FC)=\hyp_G(P)$ and $\Z(\FC)\cong\Z(G/\pcore_{p'}(G))$ by the $\Z^*$-theorem.
\end{enumerate}
\end{Ex}

\begin{Thm}[\textsc{Díaz--Glesser--Park--Stancu}]
Let $\FC$ be a saturated fusion system on $P$. 
\begin{enumerate}[(i)]
\item If $\EC\subseteq\FC$ is a saturated subsystem \textup(subcategory\textup) on $P$, then $\foc(\FC)=\foc(\EC)\iff\hyp(\FC)=\hyp(\EC)$.
\item If $P$ has no quotient isomorphic to $C_p\wr C_p$, then $\foc(\FC)=\foc(\N_\FC(P))$. In particular, $\FC$ is trivial if and only if $\Aut_\FC(P)=\Inn(P)$.
\end{enumerate}
\end{Thm}

\begin{Thm}[\textsc{Díaz--Glesser--Mazza--Park}]
Let $\FC$ be a saturated fusion system on $P$ with $p\ge 5$. Then $\foc(\FC)=\foc(\N_\FC(K))$ where $K$ is the characteristic subgroup from \autoref{glauberman}.
\end{Thm}

Kessar--Linckelmann and Onofrei--Stancu have translated Theorems~\ref{ZJ} and \ref{Stellmacher} to fusion systems, but this requires the definition of $Qd(p)$-\emph{free} fusion systems.

\section{Classification of fusion systems}

Let $\FC$ be a saturated fusion system on a finite $p$-group $P$. Let $\Out_\FC(Q):=\Aut_\FC(Q)/\Inn(Q)$ for $Q\le P$.

\begin{Thm}[\textsc{Glauberman--Thompson}]
If $\foc(\FC)=P\ne 1$ and $p\ge 5$, then $\Out_\FC(P)\ne 1$.
\end{Thm}

\begin{Def}
A subgroup $Q\le P$ is called $\FC$-\emph{essential} if
\begin{itemize}
\item $\C_P(Q)\le Q$,
\item $Q$ is $\FC$-normalized,
\item there exists a \emph{strongly $p$-embedded} subgroup $H<\Out_\FC(Q)$, i.\,e. $p\bigm||H|$ and $p\nmid |H\cap {}^xH|$ for every $x\in\Out_\FC(Q)\setminus H$ (cf. Frobenius complement).\footnote{A finite group $L$ contains a strongly $p$-embedded subgroup if and only if the graph with vertex set $\Syl_p(L)$ and edges $(S,T)\iff S\cap T\ne 1$ is disconnected.}
\end{itemize}
\end{Def}

\begin{Ex}\label{exessentail}\hfill
\begin{enumerate}[(i)]
\item\label{eqrad} Every $\FC$-essential subgroup $Q\le P$ is $\FC$-\emph{radical}, i.\,e. $\pcore_p(\Aut_\FC(Q))=\Inn(Q)$. To prove this, let $H<U:=\Out_\FC(Q)$ be strongly $p$-embedded. Let $H_p\le U_p$ be Sylow $p$-subgroups of $H$ and $U$ respectively. For $x\in\N_{U_p}(H_p)$, we have $1\ne H_p\le H\cap {}^xH$ and therefore $x\in\N_{H_p}(H_p)=H_p$. Hence, $\N_{U_p}(H_p)=H_p$ and $H_p=U_p$ by standard group theory. It follows that $\pcore_p(U)=\bigcap_{x\in U}{}^xU_p\le H\cap{}^uH$ for any $u\in U\setminus H$. So $\pcore_p(U)=1$.

\item Part \eqref{eqrad} shows that every essential subgroup $Q$ has non-trivial $p'$-automorphisms and $\Out_\FC(Q)$ acts faithfully on $Q/\Phi(Q)\cong C_p^r$. Therefore, $\Out_\FC(Q)\le\GL(r,p)$.

\item Since $P$ is $\FC$-automized, $\Out_\FC(P)$ is a $p'$-group and $P$ is not essential.
\item If $P$ is abelian, then there are no essential subgroups, since $P$ is the only self-centralizing subgroup.
\item Let $G=S_4$, $P\in\Syl_2(G)$ and $\FC=\FC_P(G)$. Then $V_4:=\langle(12)(34),(13)(24)\rangle\le P$ is $\FC$-essential since $\Out_\FC(V_4)=G/V_4\cong S_3$ contains the strongly $2$-embedded subgroup $P/V_4\cong C_2$. On the other hand, $Q:=\langle(12),(34)\rangle\cong V_4$ is not $\FC$-essential (provided $Q\le P$).
\end{enumerate}
\end{Ex}

\begin{Thm}[\textsc{Alperin--Goldschmidt}'s fusion theorem]\label{AFT}
Let $\EC$ be a set of representatives for the $\FC$-conjugacy classes of essential subgroups.
Every isomorphism in $\FC$ is a composition of isomorphisms of the form $\phi\colon S\to T$ with the following properties:
\begin{enumerate}[(i)]
\item $S,T\le Q\in\EC\cup\{P\}$.
\item $\exists\psi\in\Aut_\FC(Q)$ such that $\psi_S=\phi$,
\item If $Q\in\EC$, then $\psi$ is a $p$-element.
\end{enumerate}
\end{Thm}

The number $|\EC|$ in \autoref{AFT} is called the \emph{essential rank} of $\FC$. 

\begin{Thm}\label{bender}
A group $G$ contains a strongly $p$-embedded subgroup if and only if one of the following holds:
\begin{enumerate}[(1)]
\item\label{generic} $\pcore_p(G)=1$ and the Sylow $p$-subgroups of $G$ are (non-trivial) cyclic or (generalized) quaternion $2$-groups.
\item\label{simple} $\pcore^{p'}\bigl(G/\pcore_{p'}(G)\bigr)$ is one of the following:
\begin{itemize}
\item $\PSL(2,p^n)$ for $n\ge 2$,
\item $\PSU(3,p^n)$ for $n\ge 1$ and $(p,n)\ne(2,1)$,
\item $\Sz(2^{2n+1})$ for $p=2$ and $n\ge 1$,
\item $^2G_2(3^{2n-1})$ for $p=3$ and $n\ge 1$,
\item $A_{2p}$ for $p\ge 5$,
\item $\PSL_3(4)$, $M_{11}$ for $p=3$,
\item $\Aut(\Sz(32))$, $^2F_4(2)'$, $McL$, $Fi_{22}$ for $p=5$,
\item $J_4$ for $p=11$.
\end{itemize}
\end{enumerate}
\end{Thm}
\begin{proof}
The proof of $p=2$ is due to Bender, while the case $p>2$ was established during the CFSG.
\end{proof}

\begin{Ex}\hfill
\begin{enumerate}[(i)]
\item In the situation of \autoref{bender}\eqref{generic}, every $P\in\Syl_p(G)$ has a unique subgroup $\Omega(P)$ of order $p$. It is easy to see that $\N_G(\Omega(P))$ is strongly $p$-embedded in $G$. 

\item The groups in \autoref{bender}\eqref{simple} apart from $A_{2p}$, $^2G_2(3)\cong\PSL(2,8).3$ and $\Aut(\Sz(32))\cong\Sz(32).5$ are precisely the simple groups $G$ with a non-cyclic \emph{trivial intersection} (TI) Sylow $p$-subgroup $P$, i.\,e. $P\cap {}^gP=1$ for all $g\in G\setminus\N_G(P)$. Thus, $\N_G(P)$ is strongly $p$-embedded in this case.

\item Let $p\ge 5$ and $G=A_{2p}$. Then $H:=G\cap (S_p\wr C_2)$ is strongly $p$-embedded in $G$.
\end{enumerate}
\end{Ex}

\begin{Cor}
Let $Q\le P$ be $\FC$-essential with $p\ge 5$. Then one of the following holds for $N:=\N_P(Q)/Q$:
\begin{enumerate}[(1)]
\item $N$ is cyclic, elementary abelian or $5^{1+2}_-$ (extraspecial of exponent $25$).
\item $\exp(N)=p$ and $\Z(N)=N'=\Phi(N)\cong C_p^n$ where $|N|=p^{3n}$ \textup(i.\,e. $N$ is \emph{special}\textup).
\end{enumerate}
\end{Cor}

Alperin--Goldschmidt's fusion theorem and \autoref{bender} make it feasible to determine all saturated fusion systems on a given $p$-group. Parker--Semeraro have developed a MAGMA algorithm for this purpose and discovered fusion systems overlooked in previous work.\footnote{\url{https://github.com/chris1961parker/fusion-systems}}
Since “most” $p$-groups do not have non-trivial $p'$-automorphisms, there are very few essential subgroups and “most” fusion systems are trivial.

\begin{Def}\hfill
\begin{itemize}
\item We call $\FC$ \emph{controlled} if there are no essential subgroups.
\item We call $P$ \emph{resistant}\footnote{sometimes called \emph{Swan group}} if every fusion system on $P$ is controlled.
\item We call $P$ \emph{fusion-trivial} if every fusion system on $P$ is trivial.
\end{itemize}
\end{Def}

\begin{Ex}\label{excontrol}\hfill
\begin{enumerate}[(i)]
\item Let $P\in\Syl_p(G)$. Then $\FC_P(G)$ is controlled if and only if $\N_G(P)$ controls fusion in $P$. 

\item By the Schur--Zassenhaus theorem, $\Inn(P)$ has a complement $A$ in $\Aut_\FC(P)$ since $P$ is automized. 
If $\FC$ is controlled, then $\FC=\FC_P(P\rtimes A)$. In particular, $\FC$ is not exotic.

\item Every abelian $p$-group is resistant by \autoref{exessentail}.

\item Stancu proved that every metacyclic $p$-group for $p>2$ is resistant. I proved that metacyclic $2$-groups apart from $D_{2^n}$, $Q_{2^n}$, $SD_{2^n}$ and $C_{2^n}^2$ are fusion-trivial.

\item Every $2$-group of the form $C_{2^{a_1}}\times\ldots\times C_{2^{a_n}}$ with $a_1<\ldots<a_n$ is fusion-trivial. The smallest non-trivial fusion-trivial $p$-group of odd order is $\mathtt{SmallGroup}(3^6,46)$. 

\item\label{exconD8} Let $\FC$ be a saturated fusion system on $P=\langle x,y:x^4=y^2=1,\ {}^yx=x^{-1}\rangle\cong D_8$. There are three cases:
\begin{enumerate}[(a)]
\item $\FC$ is controlled and therefore trivial since $\Aut(P)\cong D_8$ is a $2$-group.
\item There is exactly one essential subgroup, say $\langle x^2,y\rangle$. Then $\FC=\FC_P(S_4)$.
\item There are two essential subgroups $\langle x^2,y\rangle$ and $\langle x^2,xy\rangle$. Then $\FC=\FC_P(\GL(3,2))$. In contrast to $S_4$, all involutions in $\GL(3,2)$ are conjugate, namely to the rational canonical form
\[\begin{pmatrix}
1&.&.\\
.&.&1\\
.&1&.
\end{pmatrix}.\]
\end{enumerate}
\end{enumerate}
\end{Ex}

\begin{Def}
We call $Q\unlhd P$ \emph{normal} in $\FC$ (and write $Q\unlhd\FC$) if $\N_\FC(Q)=\FC$.
\end{Def}

Let $Q,R\unlhd\FC$ and $\phi\in\Hom_\FC(S,T)$. Then there exist $\psi\in\Hom_\FC(RS,RT)$ and $\tau\in\Hom_\FC(QRS,QRT)$ such that $\psi(R)=R$, $\psi_S=\phi$, $\tau(Q)=Q$ and $\tau_{RS}=\psi$. Hence, $\tau(QR)=\tau(Q)\psi(R)=QR$ and $\tau_S=\psi_S=\phi$. This shows that $\phi\in\N_\FC(QR)$ and $QR\unlhd\FC$. The following definition is therefore justified.

\begin{Def}\hfill
\begin{itemize}
\item The (unique) largest normal subgroup of $\FC$ is denoted by $\pcore_p(\FC)$. 

\item We call $\FC$ \emph{constrained} if $\C_P(\pcore_p(\FC))\le\pcore_p(\FC)$.
\end{itemize} 
\end{Def}

\begin{Ex}\hfill
\begin{enumerate}[(i)]
\item If $Q\le P\in\Syl_p(G)$ and $Q\unlhd G$, then $Q\unlhd\FC_P(G)$. On the other hand, if $P$ is abelian, then $P\unlhd\FC_P(G)$, but not necessarily $P\unlhd G$. 

\item Every essential subgroup contains $\pcore_p(\FC)$ and $\Z(\FC)\le\pcore_p(\FC)$.

\item Every controlled fusion system $\FC$ on $P$ is constrained with $\pcore_p(\FC)=P$. On the other hand, $\FC:=\FC_{D_8}(S_4)$ is constrained with $\pcore_2(\FC)=V_4$, but not controlled.

\item Let $G=\GL(3,2)$ and $P\in\Syl_2(G)$. Then $\FC_P(G)$ is not constrained since the two essential subgroups intersect in $\Z(P)$ (cf. \autoref{excontrol}\eqref{exconD8}). Moreover, $\foc(\FC)=P\ntrianglelefteq\FC$.

\item A group $G$ is called $p$-\emph{constrained} if $\C_{\overline{G}}(\pcore_p(\overline{G}))\le\pcore_p(\overline{G})$ where $\overline{G}:=G/\pcore_{p'}(G)$. In this case $\FC:=\FC_P(G)$ is constrained with $\overline{\pcore_p(\FC)}=\overline{\pcore_p(\overline{G})}$. By \autoref{thmmodel} below every constrained fusion system arises in this way. The Hall--Higman lemma asserts that every ($p$-)solvable group is $p$-constrained.
\end{enumerate}
\end{Ex}

\begin{Thm}[Model theorem]\label{thmmodel}
For every constrained fusion system $\FC$ on $P$ there exists a unique finite group $G$ \textup(called \emph{model}\textup) such that
\begin{enumerate}[(i)]
\item $P\in\Syl_p(G)$ and $\FC=\FC_P(G)$.
\item $\pcore_{p'}(G)=1$ and $\C_G(\pcore_p(G))\le\pcore_p(G)$. 
\end{enumerate}
In particular, $\FC$ is not exotic. 
\end{Thm}

Let $G$ be a model for the constrained fusion system $\FC$ on $P$ with $|P|=p^n$. A theorem of Hall shows that
\[|G|\le|G/\pcore_p(G)||P|\le|\Aut(\pcore_p(G))|p^n\le|\GL(n,p)|p^n=(p^n-1)\ldots(p^n-p^{n-1})p^n.\]
In particular, there are only finitely many choices when $P$ is given. 

\begin{Ex}
If $\FC$ is controlled, then $P\rtimes A$ is the model for $\FC$ where $A\cong\Out_\FC(P)$ as in \autoref{excontrol}. 
\end{Ex}

\begin{Thm}[\textsc{Glesser}]\label{glesser}
Let $p>2$ and $\FC$ a non-trivial fusion system on $P$. Then $\FC$ contains \textup(as a subcategory\textup) a non-trivial constrained fusion system on $P$. 
\end{Thm}

One can use \autoref{glesser} and the model theorem to decide whether a given group $P$ is fusion-trivial. 
The fusion system $\FC_{D_{16}}(\PGL(2,7))$ (found by Craven) shows that Glesser's theorem fails for $p=2$.
In order to classify non-constrained fusion systems (especially exotic fusion systems), Oliver has introduced \emph{reduced} and \emph{tame} fusion systems. 
In an ongoing effort to simplify the CFSG, Aschbacher has investigated \emph{simple} fusion systems. Unfortunately, fusion systems of simple groups are not always simple, but well-studied nevertheless.

\section{Representation theory}

Let $F$ be an algebraically closed field of characteristic $p>0$. Let $B$ be a ($p$-)\emph{block} of $FG$, i.\,e. an indecomposable direct summand. We fix a \emph{defect group} $D\le G$ of $B$. 

\begin{Def}[\textsc{Alperin--Broué}, \textsc{Olsson}]\label{deffusbl}\hfill
\begin{itemize}
\item We call $(Q,b_Q)$ a $B$-\emph{subpair} if $Q\le D$ and $b_Q$ is a Brauer correspondent of $B$ in $Q\C_G(Q)$, i.\,e. $b_Q^G=B$. For subpairs we write $(S,b_S)\unlhd(T,b_T)$ if $S\unlhd T$ and $b_S^{T\C_G(S)}=b_T^{T\C_G(S)}$.\footnote{Alperin--Broué require additionally that $b_S$ is $T$-invariant, but Olsson showed that this is unnecessary.} Let $\le$ be the transitive closure of $\unlhd$, i.\,e. 
\[(S,b_S)\le(T,b_T)\iff(S,b_T)=(T_1,b_1)\unlhd\ldots\unlhd(T_n,b_n)=(T,b_T).\] 

\item We fix a $B$-subpair $(D,b_D)$ (by Brauer's extended first main theorem, $(D,b_D)$ is unique up to conjugation). It can be shown that for every $Q\le D$ there exists a unique subpair of the form $(Q,b_Q)\le(D,b_D)$. We fix those in the following. The fusion system $\FC=\FC_D(B)$ on $D$ is defined by
\[\Hom_\FC(S,T):=\bigl\{\phi\colon S\to T:\exists g\in G:{}^g(S,b_S)\le(T,b_T)\wedge\phi(s)={}^gs\,\forall s\in S\bigr\}.\]
\end{itemize}
\end{Def}

\begin{Thm}[\textsc{Puig}]\label{puig}
The fusion system $\FC_D(B)$ is saturated.
\end{Thm}

We call $B$ \emph{nilpotent} (\emph{controlled}, \emph{constrained}) if $\FC_D(B)$ is trivial (controlled, constrained). 
The irreducible ordinary and modular characters of $G$ can be distributed into blocks. We set $k(B):=|\Irr(B)|$ and $l(B):=|\IBr(B)|$.
Moreover, let $\foc(B):=\foc(\FC_D(B))$. 

\begin{Ex}\hfill
\begin{enumerate}[(i)]
\item The \emph{principal} block $B=B_0(G)$ contains the trivial character of $G$. In this case $D\in\Syl_p(G)$ and $\FC_D(B)=\FC_D(G)$. In particular, $G$ is $p$-nilpotent if and only if $B$ is nilpotent. In this case, all blocks of $G$ are nilpotent.

\item If $\C_G(\pcore_p(G))\le\pcore_p(G)$, then $B_0(G)$ is the only block of $G$. 

\item In the context \autoref{deffusbl}, $\Out_\FC(D)=\N_G(D,b_D)/D\C_G(D)$ is called the \emph{inertial quotient} of $B$ and its order is the \emph{inertial index}, which is coprime to $p$ by \autoref{puig}. 

\item The dihedral group $G=D_{24}$ has a nilpotent $3$-block with defect group $D\cong C_3$, while the principal $3$-block is not nilpotent. This shows that $D$ alone does not determine the fusion system of a block.  
\end{enumerate}
\end{Ex}

\begin{Con}\label{con}
For every block $B$ of $G$ with defect group $D$ there exists a finite group $H$ such that $D\in\Syl_p(H)$ and $\FC_D(B)=\FC_D(H)$.
\end{Con}

\begin{Thm}\hfill
\begin{enumerate}[(i)]
\item Let $B$ be a block of $S_n$ with defect group $D$. Then there exists an integer $w\ge 0$ \textup(called the \emph{weight} of $B$\textup) such that $D\in\Syl_p(S_{pw})$ and $\FC_D(B)=\FC_D(S_{pw})$. 

\item Let $B$ be a block of $A_n$ with defect group $D$. Then $\FC_D(B)\in\{\FC_D(S_{pw}),\FC_D(A_{pw})\}$ for some $w\ge 0$.
\end{enumerate}
\end{Thm}

\begin{Thm}[\textsc{Humphreys}, \textsc{An--Dietrich}]
Let $B$ be a block of a group $G$ of Lie type in characteristic $p$ with defect group $D$. Then $D=1$ or $D\in\Syl_p(G)$ and $\FC_D(B)=\FC_D(G)$. 
\end{Thm}

It has been shown that there is no block with the exotic fusion systems mentioned in \autoref{exexotic}. 

\begin{Thm}[\textsc{Puig}]
Let $B$ be nilpotent. Then $B\cong (FD)^{n\times n}$ for some $n\ge 1$. In particular, $B$ and $FD$ are \emph{Morita equivalent}, i.\,e. they have equivalent module categories. Moreover, $k(B)=k(D)$ and $l(B)=1$.
\end{Thm}

\begin{Thm}[\textsc{Fong--Reynolds}]
Let $b$ be a block of $N\unlhd G$ with inertial group $G_b$. Then the Brauer correspondence $C\mapsto C^G$ gives a bijection between the blocks of $G_b$ covering $b$ and the blocks of $G$ covering $b$. Moreover, $C$ and $C^G$ are Morita equivalent and have the same fusion system. 
\end{Thm}

\begin{Thm}[Second Fong Reduction]\label{Fong2}
Let $B$ be a block of $G$ covering a $G$-invariant block of $N\unlhd G$ with defect $0$. Then $B$ is Morita equivalent to a block of a finite group $H$ with the same fusion system. Moreover, there exists a cyclic $p'$-subgroup $Z\le\Z(H)$ such that $H/Z\cong G/N$. 
\end{Thm}

The block of $H$ in the situation of \autoref{Fong2} is Morita equivalent to a twisted group algebra $F_\alpha[G/N]$ where $\alpha\in\cohom^2(G/N,F^\times)$. Conversely, every such twisted group algebra is Morita equivalent to a block of a suitable central extension. If $B$ is the principal block or if $G/N$ has trivial Schur multiplier, then $\alpha=1$ and $B$ is Morita equivalent to $F[G/N]$. This applies also to the following two theorems.

\begin{Thm}[\textsc{Külshammer}]
If $D\unlhd G$, then $B$ is controlled and Morita equivalent to a twisted group algebra $F_\alpha[D\rtimes\Out_\FC(D)]$ where $\alpha\in\cohom^2(\Out_\FC(D),F^\times)$. 
\end{Thm}

\begin{Thm}[\textsc{Külshammer}]
If $G$ is $p$-solvable, then $B$ is constrained and Morita equivalent to $F_\alpha H$ where $H$ is the model for $\FC_D(B)$ from \autoref{thmmodel} and $\alpha\in\cohom^2(H,F^\times)$.
\end{Thm}

\begin{Thm}[\textsc{Eaton--Kessar--Külshammer--Sambale}]
Every $2$-block $B$ with a metacyclic defect group $D$ belongs to one of the following cases:
\begin{enumerate}[(1)]
\item $B$ is nilpotent.
\item $D$ is dihedral, semidihedral or quaternion and $B$ has \emph{tame} representation type \textup(Morita equivalence classes classified up to scalars\textup).
\item $D\cong C_{2^n}^2$ and $B$ is Morita equivalent to $F[D\rtimes C_3]$.
\item $D\cong C_2^2$ and $B$ is Morita equivalent to $B_0(A_5)$.
\end{enumerate}
\end{Thm}

\begin{Con}[Blockwise $\Z^*$-conjecture]\label{Z*con}
Let $B$ be a block with fusion system $\FC$ and $Z:=\Z(\FC)$. Then $B$ is Morita equivalent to its Brauer correspondent $b_Z$ in $\C_G(Z)$.
\end{Con}

Since $\N_G(D,b_D)\le\C_G(Z)$, $b_Z$ is indeed the unique Brauer correspondent of $B$ by Brauer's first main theorem. 
\autoref{Z*con} holds for principal blocks by \autoref{exfoc}.

\begin{Thm}[\textsc{Külshammer--Okuyama}, \textsc{Watanabe}]
In the situation of \autoref{Z*con} we have $k(B)\ge k(b_Z)$ and $l(B)\ge l(b_Z)$ with equality in both cases if $D$ is abelian.
\end{Thm}

\begin{Con}[\textsc{Rouquier}]
If $Q:=\hyp(\FC_D(B))$ is abelian, then $B$ is derived equivalent to its Brauer correspondent $B_Q$ in $\N_G(Q)$.
\end{Con}

\begin{Ex}
Suppose that $B$ has abelian defect group $D$. \emph{Broué's conjecture} predicts that $B$ and $B_Q$ are derived equivalent to their common Brauer correspondent in $\N_G(D)$. This implies Rouquier's conjecture for $B$. Conversely, if Rouquier's conjecture and the blockwise $\Z^*$-conjecture hold for $B$, then $B$ is derived equivalent to its Brauer correspondent in $\N_G(Q,b_Q)\cap\C_G(Z)=\N_G(D,b_D)$ since $D=Q\times Z$ by the Fitting decomposition (\autoref{exfoc}). Thus, Broué's conjecture holds for $B$.
\end{Ex}

\begin{Thm}[\textsc{Watanabe}]\label{watanabe}
If $Q$ is cyclic in the situation of Rouquier's conjecture, then $\FC$ is controlled with $\Out_\FC(D)\le C_{p-1}$ and
\begin{align*}
k(B)&=k(B_Q)=k(D\rtimes\Out_\FC(D)),\\
l(B)&=l(B_Q)=\lvert\Out_{\FC}(D)\rvert.
\end{align*}
\end{Thm}

If $p>2$ and $D$ is non-abelian metacyclic, then \autoref{watanabe} applies.

\begin{Def}
Let $\FC$ be a saturated fusion system on $P$ and $Q\unlhd\FC$. Then the (saturated) fusion system $\FC/Q$ on $P/Q$ consists of the morphism $\phi\colon S/Q\to T/Q$ such that there exists a morphism $\psi\colon S\to T$ in $\FC$ with $\phi(xQ)=\psi(x)Q$ for all $x\in S$.
\end{Def}

\begin{Thm}\label{local}
Let $B$ be a block of $G$ with defect group $D$ and $\FC=\FC_D(B)$. Let $(Q,b_Q)$ be a $B$-subpair such that $Q$ is $\FC$-normalized.
Then
\begin{enumerate}[(i)]
\item $b_Q$ has defect group $Q\C_D(Q)$ and fusion system $Q\C_\FC(Q)$.
\item $b_Q^{\N_G(Q)}$ has defect group $\N_D(Q)$ and fusion system $\N_\FC(Q)$. 
\item $b_Q$ dominates a unique block $\overline{b_Q}$ of $\C_G(Q)Q/Q$ with defect group $\C_D(Q)Q/Q$ and fusion system $Q\C_\FC(Q)/Q$. Moreover, $l(b_Q)=l(\overline{b_Q})$. 
\end{enumerate}
\end{Thm}

In the situation of \autoref{local} the map $S\mapsto S/Q$ is a bijection between the set of $Q\C_\FC(Q)$-essential subgroups and the set of $Q\C_\FC(Q)/Q$-essential subgroups. This allows inductive arguments.

\begin{Thm}[\textsc{Brauer}]\label{brauer}
Let $B$ be a block of $G$ with defect group $D$ and $\FC=\FC_D(B)$. Let $\mathcal{X}\subseteq D$ be a set of representatives for the $\FC$-conjugacy classes of $D$ such that $\langle x\rangle$ is $\FC$-normalized for $x\in \mathcal{X}$. Then
\[k(B)=\sum_{x\in \mathcal{X}}l(b_x)=\sum_{x\in \mathcal{X}}l(\overline{b_x}),\]
where $b_x:=b_{\langle x\rangle}$.
In particular, $k(B)-l(B)$ is locally determined. 
\end{Thm}

The fusion system of a block does not determine $k(B)$ or $l(B)$. For example, the group 
\[G=\mathtt{SmallGroup}(72,23)\cong C_3^2\rtimes D_8\] with $|\Z(G)|=2$ 
from the \emph{small groups library} has two $3$-blocks $B_0$, $B_1$ with defect group $D=C_3^2$ and fusion system $\FC_D(S_3^2)$, but $l(B_0)=4$ and $l(B_1)=1$. We need an additional ingredient: 
For an $F$-algebra $A$ let $z(A)$ be the number of simple projective $A$-modules up to isomorphism.

\begin{Con}[\textsc{Alperin}'s weight conjecture]\label{conalp}
Let $B$ be a block of $G$ with defect group $D$ and $\FC=\FC_D(B)$. Let $\mathcal{R}$ be a set of representatives for the $\FC$-conjugacy classes of self-centralizing, $\FC$-centralized subgroups of $D$.
Then
\[l(B)=\sum_{Q\in\mathcal{R}}{z\bigl(F_{\gamma_Q}\Out_{\FC}(Q)\bigr)}\]
where $\gamma_Q\in\cohom^2(\Out_{\FC}(Q),F^\times)$ is the so-called \emph{Külshammer--Puig class}.
\end{Con}

\begin{Ex}\hfill
\begin{enumerate}[(i)]
\item Suppose that $B$ is controlled in the situation of \autoref{conalp}. Then $z(F_{\gamma_Q}\Out_{\FC}(Q))=0$ for $Q<D$, since $\N_D(Q)/Q$ is a non-trivial normal $p$-subgroup of $\Out_{\FC}(Q)$. Hence, Alperin's conjecture becomes $l(B)=z(F_{\gamma_D}\Out_\FC(D))$. If in addition $B$ is the principal block (or $\Out_\FC(D)$ has trivial Schur multiplier), then $l(B)=z(F\Out_\FC(D))=k(\Out_\FC(D))$. 

\item Let $B$ be the principal $2$-block of $S_4$ with $D=\langle x,y\rangle$ as in \autoref{excontrol}. The self-centralizing, $\FC$-centralized subgroups are
$Q_1=\langle x^2,y\rangle$, $Q_2=\langle x^2,xy\rangle$, $Q_3=\langle x\rangle$ and $Q_4=D$. Alperin's conjecture becomes
\[l(B)=\sum_{i=1}^4z\bigl(F_{\gamma_{Q_i}}\Out_\FC(Q_i)\bigr)=z(FS_3)+2z(FC_2)+z(F)=1+0+1=2.\]
\end{enumerate}
\end{Ex}

\begin{Def}
The \emph{height} $h\ge 0$ of $\chi\in\Irr(B)$ is defined by $\chi(1)_p=p^h|G:D|_p$. Let $k_h(B)$ be the number of $\chi\in\Irr(B)$ with height $h$. 
\end{Def}

\begin{Thm}[\textsc{Broué--Puig}, \textsc{Robinson}]
Let $B$ be a block with defect group $D$. Then 
\begin{enumerate}[(i)]
\item $|D/\foc(B)|$ divides $k_0(B)$ with equality if and only if $B$ is nilpotent.
\item $|\Z(D)\foc(B)/\foc(B)|$ divides $k_h(B)$ for all $h\ge 0$.
\end{enumerate}
\end{Thm}

If $D$ is abelian, then $|\Z(\FC_D(B))|$ divides $k(B)$, because $D=\foc(B)\times\Z(\FC_D(B))$. 

\emph{Dade's conjecture}, expressing $k_h(B)$ in terms of alternating sums, has been reformulated in terms of fusion systems by Robinson (\emph{ordinary weight conjecture}). 
Kessar--Linckelmann--Lynd--Semeraro have generalized this and other conjectures in block theory to statements on abstract fusion systems.

\begin{small}

\end{small}

\end{document}